\documentclass[11 pt]{amsart}
\usepackage{latexsym,amscd,amssymb, graphicx, amsthm,hyperref}
\usepackage{tikz}

\usepackage[margin=1in]{geometry}

\numberwithin{equation}{section}

\newtheorem{theorem}{Theorem}[section]

\newtheorem{corollary}[theorem]{Corollary}
\newtheorem{lemma}[theorem]{Lemma}

\newtheorem{problem}[theorem]{Problem}

\newtheorem{definition}[theorem]{Definition}
\theoremstyle{definition}

\newcommand{\Hilb}{\mathrm {Hilb}}
\newcommand{\symm}{\mathfrak{S}}
\newcommand{\grFrob}{\mathrm {grFrob}}
\newcommand{\Frob}{\mathrm {Frob}}

\newcommand{\CC}{\mathbb {C}}
\newcommand{\QQ}{\mathbb {Q}}
\newcommand{\ZZ}{\mathbb {Z}}

\newcommand{\OP}{\mathcal{OP}}

\newcommand{\WWW}{\mathcal{W}}

\newcommand{\xx}{\mathbf {x}}

\begin{document}

\title[Packed words and quotient rings]
{Packed words and quotient rings}

\author{Dani\"el Kroes and Brendon Rhoades}
\address
{Department of Mathematics \newline \indent
University of California, San Diego \newline \indent
La Jolla, CA, 92093-0112, USA}
\email{dkroes@ucsd.edu, bprhoades@ucsd.edu}

\begin{abstract}
The coinvariant algebra is a quotient of the polynomial ring $\QQ[x_1,\ldots,x_n]$ whose algebraic properties are governed by the combinatorics of permutations of 
length $n$. A word $w = w_1 \dots w_n$ over the positive integers is {\em packed} if whenever $i > 2$ appears as a letter of $w$, so does $i-1$. We introduce 
a quotient $S_n$ of $\QQ[x_1,\ldots,x_n]$ which is governed by the combinatorics of packed words. 
We relate our quotient $S_n$ to the generalized coinvariant rings of Haglund, Rhoades, and Shimozono as well as the 
superspace coinvariant ring.
\end{abstract}

\maketitle

\section{Introduction}
Consider the polynomial ring $\QQ[\xx_n] := \QQ[x_1,\ldots,x_n]$ in $n$ variables. The symmetric group $\symm_n$ acts on $\QQ[\xx_n]$ by variable permutation. It is known that the corresponding invariant subring $\QQ[\xx_n]^{\symm_n}$ of \emph{symmetric functions} has algebraically independent homogeneous generators $e_1(\xx_n),\ldots,e_n(\xx_n)$, where
\[
e_d := \sum_{1 \leq i_1 < \ldots < i_d \leq n} x_{i_1} \cdots x_{i_d}
\]
is the degree $d$ \emph{elementary symmetric polynomial}. The \emph{invariant ideal} is the ideal generated by the symmetric functions with constant term:
\[
I_n := \langle \QQ[\xx_n]_+^{\symm_n} \rangle = \langle e_1,\ldots,e_n \rangle.
\]
The \emph{coinvariant algebra} $R_n := \QQ[x_n]/I_n$ has long been studied. 
In particular, as ungraded $\symm_n$-module $R_n \cong_{\symm_n} \QQ[\symm_n]$ coincides with the regular representation of $\symm_n$. 
Moreover, the Hilbert series $\Hilb(R_n;q) = (1+q)(1+q+q^2)\cdots(1+q+\ldots+q^{n-1})$ coincides with the 
generating function of both the inversion and major index statistic on permutations.

Traditionally studied in physics, the superspace ring $\Omega_n$ has received significant recent attention in coinvariant theory
\cite{RW-module, Zabrocki}.
For a positive integer $n$, {\em superspace} of rank $n$ is the tensor product
\[
\Omega_n := \QQ[x_1, \dots, x_n] \otimes \wedge \{ \theta_1, \dots , \theta_n \}
\]
of a rank $n$ polynomial ring with a rank $n$ exterior algebra.
The group $\symm_n$ acts diagonally on $\Omega_n$, viz. $w.x_i := x_{w(i)}$,
$w.\theta_i := \theta_{w(i)}$. Let $(\Omega_n)^{\symm_n}_+ \subseteq \Omega_n$ denote the space of $\symm_n$-invariants with vanishing constant term and 
let $\langle (\Omega_n)^{\symm_n}_+ \rangle \subseteq \Omega_n$ be the ideal generated by this subspace.
Considering commuting and anticommuting variables separately, the {\em superspace coinvariant ring}
$\Omega_n/\langle (\Omega_n)^{\symm_n}_+ \rangle$ carries a bigraded action of $\symm_n$.
One recovers the classical coinvariant algebra by setting the $\theta$-variables to zero.

The
Combinatorics Group
and the Fields Institute conjectured (see \cite{Zabrocki}) a formula for the bigraded $\symm_n$-Frobenius image type of the superspace coinvariant ring. 
Let $\mathrm{SYT}(n)$ be the family of standard Young tableaux with $n$ boxes. Given $T \in \mathrm{SYT}(n)$, let $\mathrm{des}(T)$ be the number 
of descents in $T$ and let $\mathrm{maj}(T)$ be its major index. We use the $q$-analog notation
\begin{equation*}
    [n]_q := 1 + q + \cdots + q^{n-1}, \quad 
    [n]!_q := [n]_q [n-1]_q \cdots [1]_q, \quad
    {n \brack k}_q := \frac{[n]!_q}{[k]!_q \cdot [n-k]!_q}.
\end{equation*}
The Fields Group conjectured \cite{Zabrocki} that 
\begin{equation}
    \label{superspace-frobenius}
     \grFrob( \Omega_n/\langle (\Omega_n)^{\symm_n}_+ \rangle; q, z) = \sum_{k = 1}^n z^{n-k} \cdot 
     C_{n,k}(\xx;q)
\end{equation}
where
\begin{equation}
\label{c-function-definition}
    C_{n,k}(\xx; q) := 
    \sum_{T \in \mathrm{SYT}(n)} q^{\mathrm{maj}(T) + {n-k \choose 2} + (n-k) \cdot \mathrm{des}(T)} {\mathrm{des}(T) \brack n-k}_q s_{\mathrm{shape}(T)'}(\xx).
\end{equation}
Here $s_{\mathrm{shape}(T)'}(\xx)$ is the Schur function corresponding to the conjugate of the shape of $T$.

The symmetric function $C_{n,k}(\xx;q)$ appearing in Equation~\eqref{superspace-frobenius} has appeared in the literature before.
Combinatorially, it is the $t = 0$ specialization of the function $\Delta'_{e_{k-1}} e_n$ appearing in the Haglund-Remmel-Wilson {\em Delta Conjecture}
\cite{HRW, DM}.
Algebraically, it is (up to a minor twist) the graded Frobenius image of the generalized coinvariant algebras $R_{n,k}$ introduced by 
Haglund-Rhoades-Shimozono \cite{HRS}. 
Geometrically, it is (up to the same minor twist) the graded Frobenius image of the cohomology representation afforded by the $\symm_n$-action on the 
variety $X_{n,k}$ of $n$-tuples
of lines $(\ell_1, \dots, \ell_n)$ spanning $\mathbb{C}^k$ \cite{PR}.
Despite these varied interpretations of $C_{n,k}(\xx;q)$, the formula \eqref{superspace-frobenius} remains conjectural as of this writing.

Whereas algebraic properties of the classical coinvariant ring are governed by permutations, the superspace coinvariants appear to be 
governed by packed words. A word $w = w_1 \dots w_n$ over the positive integers is {\em packed} if, for all $i > 0$,
whenever $i+1$ appears as a letter in $w$, so does $i$. Let $\WWW_n$ be the family of packed words of length $n$. For example, we have
\begin{equation*}
    \WWW_3 = \{123, 213, 132, 231, 312, 321, 112, 121, 211, 122, 212, 221, 111\}.
\end{equation*}
Packed words in $\WWW_n$ are in natural bijection with the family $\OP_n$ of all ordered set partitions of $[n]$ and have
appeared in various settings including Hopf algebras \cite{NT} and polytopes \cite{CL}.

The symmetric group $\symm_n$ acts on the set $\WWW_n$ by letter permutation. The conjecture \eqref{superspace-frobenius} implies that
\begin{equation}
    \label{ungraded-superspace}
    \Omega_n/\langle (\Omega_n)^{\symm_n}_+ \rangle \cong \QQ[\WWW_n] \otimes \mathrm{sign}
\end{equation}
as ungraded $\symm_n$-modules, where $\mathrm{sign}$ is the 1-dimensional sign representation of $\symm_n$. 
Proving the isomorphism \eqref{ungraded-superspace} remains an open problem. 
Even the dimension equality $\dim \Omega_n/\langle (\Omega_n)^{\symm_n}_+ \rangle = |\WWW_n|$ is presently out of reach.

Motivated by the conjecture \eqref{superspace-frobenius}, we
define the following family of singly-graded $\symm_n$-modules $S_n$ 
which provably have vector space dimension $|\WWW_n|$ and satisfy algebraic properties
 similar to \eqref{ungraded-superspace} and \eqref{superspace-frobenius}.
 We let $e_d^{(i)} := e_d(x_1, \dots, x_{i-1}, x_{i+1}, \dots, x_n)$ be the degree $d$ elementary symmetric polynomial
 with the variable $x_i$ omitted.

\begin{definition} \label{sn-quotient-definition}
Let $J_n \subseteq \QQ[\xx_n]$ be the ideal 
\[
J_n = \langle x_i^d \cdot e_{n-r}^{(i)} \ : \ 1 \leq i \leq n, \, 1 \leq r \leq d \rangle
\]
and let $S_n := \QQ[\xx_n]/J_n$ be the corresponding quotient ring.
\end{definition}

By convention, the degree 0 elementary symmetric polynomial is 1, so that $J_n$ contains the variable powers $x_i^n$. Additionally, we use the convention that $e_d \equiv 0$ for $d < 0$.

Although the generators of the ideal $J_n$ may appear unusual, they will arise naturally from the perspective of orbit harmonics as follows. More precisely, suppose $X \subseteq \QQ^n$ is a finite locus of points. Consider the ideal
\begin{equation}
    \mathbf{I}(X) := \{ f \in \QQ[\xx_n] \,:\, f(\mathbf{x}) = 0 \text{ for all $\mathbf{x} \in X$} \}
\end{equation}
of polynomials in $\QQ[\xx_n]$ which vanish on $X$ and let 
\begin{equation}
    \mathbf{T}(X) := \langle \tau(f) \,:\, f \in \mathbf{I}(X) - \{0\} \rangle,
\end{equation}
where $\tau(f)$ denotes the highest degree component of a nonzero polynomial $f \in \QQ[\xx_n]$.
The homogeneous ideal $\mathbf{T}(X)$ is the {\em associated graded} ideal of $\mathbf{I}(X)$ and we have isomorphisms of $\QQ$-vector spaces
\begin{equation}
    \QQ[X] \cong \QQ[\xx_n]/\mathbf{I}(X) \cong \QQ[\xx_n]/\mathbf{T}(X)
\end{equation}
which are isomorphisms of ungraded $\symm_n$-modules when $X$ is closed under the natural action of $\symm_n$ on $\QQ^n$;
the quotient $\QQ[\xx_n]/\mathbf{T}(X)$ has the additional structure of a graded $\symm_n$-module.

Given $n$ distinct rational parameters $\alpha_1, \dots, \alpha_n$, we have a natural point locus $X_n \subseteq \QQ^n$ in bijection
with $\WWW_n$, namely
\begin{equation}
    X_n = \{ (\beta_1, \dots, \beta_n) \in \QQ^n \,:\, \{ \beta_1, \dots, \beta_n \} = \{\alpha_1, \dots, \alpha_k\} \text{ for some $k$} \}.
\end{equation}
It will develop that
\begin{equation}
\label{harmonic-identification}
   \mathbf{T}(X_n) = J_n.
\end{equation}
In other words, the quotient $S_n = \QQ[\xx_n]/\mathbf{T}(X_n)$ is the graded quotient of $\QQ[\xx_n]$ arising from the packed word 
locus $X_n$.
Equation~\eqref{harmonic-identification} may be viewed as a more natural, if less computationally useful,
alternative to Definition~\ref{sn-quotient-definition}.
We prove the following facts regarding the module $S_n$.
\begin{itemize}
\item
The ungraded $\symm_n$-structure of $S_n$ coincides with the natural $\symm_n$-action on $\WWW_n$ (without sign twist) 
\begin{equation}
    \label{ungraded-sn-structure}
    S_n \cong \QQ[\WWW_n].
\end{equation}
\item
The graded $\symm_n$-structure 
is described by
\begin{equation}
    \label{sn-graded-frobenius-image}
     \grFrob(S_n; q) = \sum_{k = 1}^n q^{n-k} \cdot (\mathrm{rev}_q \circ \omega) C_{n,k}(\xx;q).
\end{equation}
Here $\mathrm{rev}_q$ is the operator on polynomials in $q$ which reverses their coefficient sequences and $\omega$ is the  symmetric function
involution which trades $e_n(\xx)$ for $h_n(\xx)$.
\end{itemize}
Finding an algebraic explanation for 
the similarity between the provable \eqref{sn-graded-frobenius-image} and the conjectural \eqref{superspace-frobenius} could shed light on a proof
of \eqref{superspace-frobenius}.

The outline of the paper is as follows. In Section \ref{sec-background} we cover some of the necessary background, including symmetric functions, representation theory of the symmetric group, and Gr\"obner theory. In Section \ref{sec-combinatorics} we describe a bijection between ordered set partitions in $\OP_n$
and certain sequences $(c_1, \dots, c_n)$ of nonnegative integers. This bijection 
will translate to a bijection between a monomial basis of $S_{n}$ and $\WWW_{n}$. In Section \ref{sec-algebra} we use this bijection to prove
\eqref{ungraded-sn-structure} and its graded refinement  
Equation~\eqref{sn-graded-frobenius-image}. In Section \ref{sec-conclusion} we end with some concluding remarks and open questions.

\section{Background}\label{sec-background}

\subsection{Symmetric functions and the representation theory of $\symm_n$} 
A \emph{partition} $\lambda$ of size $n$, denoted $\lambda \vdash n$, is a sequence $\lambda = (\lambda_1,\ldots,\lambda_m)$ of integers $\lambda_1 \geq \ldots \geq \lambda_m > 0$ with $\lambda_1+\ldots+\lambda_m=n$.

Let $\xx = (x_1,x_2,x_3,\ldots)$ be an infinite set of variables and let $\Lambda \subseteq \QQ[[\xx]]$ be the \emph{ring of symmetric functions}. It is known that the degree $n$ homogeneous piece of $\Lambda$ has a basis given by the \emph{Schur functions} $s_{\lambda}(\xx)$ 
where the index $\lambda$ ranges over all partitions of size $n$. 
For thorough definitions of the ring of symmetric functions and Schur functions, we refer to \cite{Sagan}.

We will now recall the fundamentals of the representation theory of $\symm_n$. The irreducible representations of $\symm_n$ are naturally in bijections with the partitions $\lambda \vdash n$. For every such $\lambda$, the corresponding irreducible $\symm_n$-module is denoted by $S^{\lambda}$. 

Consequently, every $\symm_n$-module $V$ decomposes as
\[
V = \bigoplus_{\lambda \vdash n} \left(S^{\lambda}\right)^{c_{\lambda}}
\]
for some integers $c_{\lambda} \geq 0$. The \emph{Frobenius character} of $V$ is the symmetric function
\[
\Frob(V) = \sum_{\lambda \vdash n} c_{\lambda} \cdot s_{\lambda}(\xx).
\]

Lastly, let $V$ be a graded vector space such that for every $d \geq 0$ the degree $d$ homogeneous component $V_d$ is finite dimensional. The \emph{Hilbert series} of $V$ is the power series in $q$ given by
\[
\Hilb(V;q) = \sum_{d \geq 0} \dim(V_d) \cdot q^d.
\]
If further $V$ carries a graded $\symm_n$-action, we define the \emph{graded Frobenius character} by
\[
\grFrob(V;q) = \sum_{d \geq 0} \Frob(V_d) \cdot q^d.
\]

More details on the representation theory of $\symm_n$ can be found in \cite{Sagan}.

\subsection{Gr\"obner theory}
In this section we will review some of the Gr\"obner theory used in this paper. The main starting point of Gr\"obner theory is a polynomial ring $k[\xx_n]$ over a field $k$ equipped with a total order $<$ on its monomials that satisfies:
\begin{itemize}
    \item[1.] $1 \leq m$ for any monomial $m$;
    \item[2.] for monomials $m_1 < m_2$ and any monomial $m$ we have $m \cdot m_1 < m \cdot m_2$. 
\end{itemize}
Such a total order is called a \emph{monomial order}. Given a monomial order $<$, for any $0 \neq f \in k[\xx_n]$ we define the \emph{leading monomial} $\mathrm{LM}(f)$ as the monomial $m$ such that $m$ has nonzero coefficient in $f$ and $n \leq m$ for any monomial $n$ with nonzero coefficient in $f$. For an ideal $I \leq k[\xx_n]$ we set $\mathrm{LM}(I)$ as the ideal generated by the leading monomials of all non-zero $f \in I$.

We know \cite{CLOS} that a basis for the $k$-vector space $k[\xx_n]/I$ is given by all monomials $m$ that do not belong to $\mathrm{LM}(I)$, which is equivalent to $m$ not being divisible by any monomial of the form $LM(f)$ with $f \in I$. This basis for $k[\xx_n]/I$ is called the \emph{standard monomial basis} with respect to the total order $<$.

The monomial order on $\QQ[\xx_n]$ used in this paper is the \emph{lexicographic order}. In this order, two monomials $m_1 = x_1^{a_1} \cdots x_n^{a_n}$ and $m_2 = x_1^{b_1} \cdots x_n^{b_n}$ are compared as follows. Assuming $m_1 \neq m_2$, let $j \in \{1,2,\ldots,n\}$ be minimal such that $a_j \neq b_j$, then $m_1 < m_2$ if and only if $a_j < b_j$.

\section{The combinatorial bijection} \label{sec-combinatorics}
In this section we will establish a bijection between ordered set partitions and coinversion codes. The starting point will be a bijection established by Rhoades and Wilson \cite[Thm. 2.2]{RW}. 
Let $\OP_{n,k}$ be the family of ordered set partitions $(B_1 \mid \cdots \mid B_k)$ of $[n]$ into $k$ blocks.
Given an ordered set partition $\sigma = (B_1 \ | \ \cdots \ | \ B_k) \in \OP_{n,k}$, 
define a sequence $\mathbf{code}(\sigma) = (c_1,\ldots,c_n)$ as follows. If $1 \leq i \leq n$ and $i \in B_j$, then
\[
c_i = 
\begin{cases}
|\{\ell > j \ : \ \min(B_\ell) > i\}| & \text{if } i = \min(B_j); \\
|\{\ell > j \ : \ \min(B_\ell) > i\}| + (j-1) & \text{if } i \neq \min(B_j).
\end{cases}
\]
The sequence $\mathbf{code}(\sigma)$ was called the {\em coinversion code} of $\sigma$ in \cite{RW}.
This is a variant of the classical {\em Lehmer code} on permutations in the case $k = n$.

The coinversion $\mathbf{code}(\sigma)$ of ordered set partitions $\sigma \in \OP_{n,k}$ were characterized in \cite{RW} as follows.
Given a subset $S = \{s_1 < \cdots < s_d \} \subseteq [n]$, define the {\em skip sequence} by $\gamma(S) = (\gamma_1, \dots, \gamma_n)$ where
\begin{equation}
    \gamma_i = \begin{cases}
    i -j + 1 & \text{if $i = s_j \in S$} \\
    0 & \text{if $i \notin S$.}
    \end{cases}
\end{equation}
Also let $\gamma(S)^* = (\gamma_n, \dots, \gamma_1)$ be the {\em reverse skip sequence}. For example, if $n = 7$ and $S = \{2,3,6\}$
we have $\gamma(S) = (0,2,2,0,0,4,0)$ and $\gamma(S)^* = (0,4,0,0,2,2,0).$

\begin{theorem} (\cite[Thm. 2.2]{RW})
\label{RW-bijection}
Let $1 \leq k \leq n$. The map $\sigma \mapsto \mathbf{code}(\sigma)$ is a bijection from ordered set partitions of $[n]$
with $k$ blocks to the family of nonnegative integer sequences $(c_1,\ldots,c_n)$ such that 
\begin{itemize}
    \item for all $1 \leq i \leq n$ we have $c_i < k$,
    \item for any subset $S \subseteq [n]$ with $|S| = n-k+1$, the componentwise inequality $\gamma(S)^* \leq (c_1,\ldots,c_n)$ fails to hold.
\end{itemize}
\end{theorem}

For future reference we recall the inverse map introduced in the proof of the above theorem. This inverse map uses the following insertion procedure.

For $(B_1 \ | \ \cdots \ | \ B_k)$ a sequence of $k$ (possibly empty) sets of positive integers we define the \emph{coinversion labels} as follows. First, label the empty sets $0,1,\ldots,j$ from right to left, and then label the nonempty sets $j+1,\ldots,j+k-1$ from left to right. 

For a sequence $(c_1,\ldots,c_n)$ satisfying the conditions in 
Theorem~\ref{RW-bijection}, we construct an ordered set partition as follows. Start with a sequence $(\emptyset \ | \ \cdots \ | \ \emptyset)$ of $k$ copies of the empty set, and for $i = 1,2,\ldots,n$ insert the number $i$ in the block with label $c_i$ under the coinversion labeling.

For example, let $n = 7$, $k = 4$ and consider the sequence $c = (2,1,2,0,2,0,2)$. The resulting ordered set partition will be $(6 \ | \ 13 \ | \ 257 \ | \ 4)$, as shown by the following process, starting with the labeled sequence of blocks $(\emptyset^3 \ | \ \emptyset^2 \ | \ \emptyset^1 \ | \ \emptyset^0)$. 
\[
\begin{tabular}{c|c|c}
    $i$ & $c_i$ & updated labeled sequence of blocks \\
\hline
    $1$ & $2$ & $(\emptyset^2 \ | \ 1^3 \ | \ \emptyset^1 \ | \ \emptyset^0)$ \\
    $2$ & $1$ & $(\emptyset^1 \ | \ 1^2 \ | \ 2^3 \ | \ \emptyset^0)$ \\
    $3$ & $2$ & $(\emptyset^1 \ | \ 13^2 \ | \ 2^3 \ | \ \emptyset^0)$ \\
    $4$ & $0$ & $(\emptyset^0 \ | \ 13^1 \ | \ 2^2 \ | \ 4^3)$ \\
    $5$ & $2$ & $(\emptyset^0 \ | \ 13^1 \ | \ 25^2 \ | \ 4^3)$ \\
    $6$ & $0$ & $(6^0 \ | \ 13^1 \ | \ 25^2 \ | \ 4^3)$ \\
    $7$ & $2$ & $(6^0 \ | \ 13^1 \ | \ 257^2 \ | \ 4^3)$
\end{tabular}
\]

In our algebraic analysis of $S_n$
we will need a version of this insertion which maps the family
 of ordered set partitions of $[n]$ with \emph{at least} $k$ blocks bijectively onto
 a certain collection $(c_1, \dots, c_n)$ of length $n$ `code words' over the nonnegative integers.
In the bijection $\mathbf{code}$ of 
Theorem~\ref{RW-bijection}, the ordered set partition $(1|2|\cdots|m,m+1,\ldots,n)$ has code $(0,0,\ldots,0)$ for any number of blocks $m$, so 
we cannot simply take the union of these maps for $m \geq k$.

We resolve the problem in the above paragraph by working
with a different version of the coinversion code, which we will call the \emph{boosted coinversion code}. For an ordered set partition
$\sigma = (B_1 \ | \ \cdots \ | \ B_k)$ we define $\overline{\mathbf{code}}(\sigma) = (c_1,\ldots,c_n)$ as follows. Suppose $1 \leq i \leq n$ and $i \in B_j$, then
\[
c_i = 
\begin{cases}
|\{\ell > j \ : \ \min(B_\ell) > i\}| & \text{if } i = \min(B_j); \\
|\{\ell > j \ : \ \min(B_\ell) > i\}| + j & \text{if } i \neq \min(B_j).
\end{cases}
\]
Compared to the coinversion codes from before, the difference is that all the numbers corresponding to non-minimal elements of blocks are raised by one, and we say that these numbers are \emph{boosted}.

The remainder of the section will be devoted to the proof of the following theorem.

\begin{theorem}\label{thm-combinatorial-bijection}
Let $1 \leq k \leq n$. The map $\sigma \mapsto \overline{\mathbf{code}}(\sigma)$ is a bijection from the set of ordered set partitions of $[n]$ with
at least $k$ blocks to the family of nonnegative integer sequences such that
\begin{itemize}
    \item for all $1 \leq i \leq n$ we have $c_i < n$. 
    \item for any subset $S \subseteq [n]$ with $|S| = n-k+1$ the componentwise inequality $\gamma(S)^* \leq (c_1,\ldots,c_n)$ fails to hold. 
    \item for any $1 \leq i,d \leq n$ and any $T \subseteq[n-1]$ with $|T| = n-d$ and $\gamma(T)^* = (\gamma_{n-1},\ldots,\gamma_1)$, the componentwise inequality $(\gamma_{n-1},\ldots,\gamma_i,d,\gamma_{i-1},\ldots,\gamma_1) \leq (c_1,\ldots,c_n)$ fails to hold.
\end{itemize}
\end{theorem}

The proof of the necessity of these conditions will be similar to that of the proof of \cite[Thm.2.2]{RW}. For the sufficiency of the 
conditions we use an insertion map similar to that considered above.
We begin by showing that both the number of blocks of an ordered set partition of $[n]$, as well as its classical coinversion code, can be 
recovered from its boosted coinversion code.

\begin{lemma}
Let $\sigma$ be an ordered set partition of $[n]$. Given the boosted coinversion code $\overline{\mathbf{code}}(\sigma)$ one can recover the coinversion code $\mathbf{code}(\sigma)$, as well as the number of blocks of $\sigma$. 
\end{lemma}

\begin{proof}
Note that the second part is immediate once we have recovered $\mathbf{code}(\sigma)$, as the number of blocks will be equal to the number of unboosted numbers, which is easily found by comparing $\overline{\mathbf{code}}(\sigma)$ and $\mathbf{code}(\sigma)$.

Given a boosted coinversion code $(c_1,\ldots,c_n)$ corresponding to an ordered set partition with $\ell$ blocks (where $\ell$ is unknown), we can think of creating the ordered set partition by following the same procedure as described before, with the only difference that the labels of all the nonempty blocks should be raised by one. 

No matter what, at some point we will fill in the last nonempty block with some number $i$, which necessarily has $c_i = 0$. Additionally, from the boosting, it is clear that $c_j > 0$ for all $j > i$, hence we can recover $i$ by looking for the last entry in our sequence that equals $0$. 

Now, assume that we have identified that $i_1 < \ldots < i_j$ are minimal in their block and that all other numbers in $[i_1,n]$ are not minimal in their block. If $i_1 = 1$ we are done. Otherwise, it is clear that we must have at least $j+1$ blocks (as clearly $1$ will be minimal in its block). Now, let $i_0 < i_1$ be the largest number that is also minimal in its block. By the inverse map, this must correspond to some index with $c_{i_0} \leq j$, as at the time of inserting $i_0$ there are exactly $j+1$ empty blocks, labeled $0,1,\ldots,j$. Additionally, for any $i_0 < i < i_1$, at the time of insertion there will be exactly $j$ empty blocks, hence the coinversion label of $i$ will be at least $j+1$ (because of the boosting). Therefore, given $\overline{\mathbf{code}}(\sigma)$ we can recognize $i_0$ as the largest index $i_0 < i_1$ with $c_{i_0} \leq j$. By induction we are done.
\end{proof}

Explicitly, the procedure above is as follows. Given a sequence $(c_1,\ldots,c_n)$, trace the sequence from right to left, marking the first $0$, then the first $0$ or $1$, then the first $0$, $1$ or $2$, etcetera. Now, decrease all the unmarked numbers by $1$ and one recovers the coinversion code. We call this procedure the \emph{unboosting} of a sequence $(c_1,\ldots,c_n)$. 

As an example, consider the boosted coinversion code $c = (2,4,2,4,0,0,1,4)$. Working from right to left we mark $c_6$ as it is the first $0$, then $c_5$ as it is at most $1$, then $c_3$ as it is the next number at most $2$ and finally $c_1$ as it is the next number that is at most $3$. Therefore, the number of blocks is equal to $4$ and the unboosted coinversion code is given by $(2,3,2,3,0,0,0,3)$. Applying the earlier bijection this coinversion code corresponds to $(37 \ | \ 124 \ | \ 6 \ | \ 58)$.

We are now ready to prove the main result of this section. 

\begin{proof}[Proof of Theorem \ref{thm-combinatorial-bijection}]
We first prove the necessity of the conditions.
Let $\sigma$ be an ordered set partition of $[n]$ with at least $k$ blocks and let $\overline{\mathbf{code}}(\sigma) = (c_1, \dots, c_n)$ be its
boosted coinversion code.
\begin{itemize}
    \item If $i$ is minimal in its block, $c_i$ will be at most the number of blocks following the block containing $i$, which is at most $n-1$. If $i$ is not minimal we have at most $n-1$ blocks, and if $i \in B_j$ we have
    \[
    c_i = j + |\{\ell > j \ : \ \min(B_\ell)>i\}| \leq j + |\{\ell > j \ : \ \text{the $\ell^{\text{th}}$ block exists}\}| \leq n-1.
    \]
    \item Suppose $S = \{n+1-t_{n-k+1},\ldots,n+1-t_1\}$ (with $t_1 < \ldots < t_{n-k+1}$) satisfies $\gamma(S)^* \leq (c_1,\ldots,c_n)$. We show
    that none of the numbers $\{t_1,\ldots,t_{n-k+1}\}$ is minimal in its block of $\sigma$, contradicting that $\sigma$ has at least $k$ blocks. 
    
    If $t_{n-k+1}$ is minimal in its block, then
    \begin{align*}
    c_{t_{n-k+1}} &= |\{ \ell > t_{n-k+1} : \begin{smallmatrix} 
    \ell \text{ is minimal in its block and} \\ \text{occurs to the right of } t_{n-k+1} \text{ in } \sigma \end{smallmatrix} \}|    \\
        &\leq |\{t_{n-k+1}+1,\ldots,n-1,n\}| = n - t_{n-k+1}.
    \end{align*}
    However, the term in $\gamma(S)^*$ in position $t_{n-k+1}$ equals $n-t_{n-k+1} + 1$, hence we conclude that $t_{n-k+1}$ is not minimal in its block. 
    
    Now, if $t_{n-k}$ were minimal in its block, we would have
    \begin{align*}
    c_{t_{n-k}} &= |\{ \ell > t_{n-k} : \ell \text{ is minimal in its block and occurs to the right of } t_{n-k} \text{ in } \sigma \}| \\
    &\leq |\{t_{n-k}+1,\ldots,n-1,n\}-\{t_{n-k+1}\}| = n - t_{n-k} - 1.
    \end{align*}
    But again, the term in $\gamma(S)^*$ in position $t_{n-k}$ equals $n - t_{n-k}$, which shows that $t_{n-k}$ cannot be minimal in its block either. An inductive argument now shows that none of $\{t_1,\ldots,t_{n-k+1}\}$ is minimal in its block.
    \item 
    For $d = n$ this is equivalent to the fact that $c_i < n$ for all $i$, so assume $1 \leq d < n$. 
    Assume for contradiction that $(\gamma_{n-1},\ldots,\gamma_{i},d,\gamma_{i-1},\ldots,\gamma_1) \leq (c_1,\ldots,c_n)$ where
    $(\gamma_{n-1}, \dots, \gamma_1) = \gamma(T)^*$ for some $T \subseteq [n-1]$ of size $|T| = n-d$. 
    Since $c_{n+1-i} \geq d$, this implies that $\sigma$ has at least $d$ blocks. 
    Let $T = \{i_1 < \ldots < i_t \leq n+1-i < i_{t+1} < \ldots < i_{t+s}\}$. 
    By the same argument used in the previous bullet,
    we see that all $n+1-i_j$ with $j \leq t$ are not minimal in their block. Now we consider two cases.
    \begin{itemize}
        \item If $n+1-i$ is not minimal in its block either, we can continue the argument as in the previous case to show that none of $n+1-i_j$ is minimal in its block. In particular we have $1 + (n-d)$ elements that are not minimal in their respective blocks, contradicting the fact that $\sigma$ has at least $d$ blocks.
        \item Now suppose that $n+1-i$ is minimal in its block. Since $c_{n+1-i} = d$, this implies that among $\{n+2-i,\ldots,n\}$ 
        at least $d$ numbers are also minimal in their respective blocks. In particular, 
        there are at least $d$ numbers that are not of the form $n+1-j$ with $j \in T$.
        But this implies that $T$ has size at most $(n-1) - d < n-d$, which is a contradiction.
    \end{itemize}
\end{itemize}

Now, we show that these conditions are sufficient. Given a sequence $(c_1,\ldots,c_n)$ we can first unboost the sequence (as we can apply this procedure to every sequence of nonnegative integers) to determine how many blocks our intended ordered set partition must have. Given this extra information, we can basically run the same inverse map as before, with the exception that we should increase the label of every nonempty block by $1$. It now suffices to check that we don't run into any troubles by doing so. Our proof will go through the following steps.
\begin{itemize}
    \item First we will show that the unboosting procedure concludes that there are at most $n-k$ boosted numbers, as this will ensure that the ordered set partition we aim for has at least $k$ blocks.
    \item Then we will inductively show that can basically run the same inverse map as before.
    \begin{itemize}
        \item First we show that the conclusion of the unboosting is that $1$ is unboosted, ensuring we have enough blocks to insert $1$ as a minimal element in its block.
        \item After that we will show that if the first $j-1$ numbers have been placed, we can place $j$ following the appropriate procedure. This argument will depend on whether $j$ is supposed to be minimal in its block or not (something that we know from the unboosting procedure).
    \end{itemize}
\end{itemize}

We will now prove each of these steps. 

\begin{itemize}
    \item Assume that we have $t$ boosted numbers $c_{n+1-i_j}$ (with $i_1 < \ldots < i_t$) and assume that $t \geq n-k+1$. Let $S = \{i_1,\ldots,i_{n-k+1}\}$, then we claim that $(c_1,\ldots,c_n) \geq \gamma(S)^*$. If $i \not \in S$, we have $\gamma(S)*_{n+1-i} = 0$, so $c_{n+1-i} \geq \gamma(S)^*_{n+1-i}$ indeed holds. Furthermore, for $i = i_j$ by assumption on $c_{n+1-i_j}$ there are $(i_j-j)$ unboosted numbers to the right of $n+1-i_j$. Therefore, since $c_{n+1-i_j}$ was boosted, we have $c_{n+1-i_j} \geq i_j-j+1 = \gamma(S)^*_{n+1-i_j}$, as desired.
    \item As mentioned before, we now show that we can run the inverse map without any issues.
        \begin{itemize}
        \item If $c_1 = 0$ it is clear that we can insert $1$, so assume $c_1 = d$ with $1 \leq d \leq n-1$. Our goal is to show that in the unboosting procedure we conclude that $1$ has to be minimal in its block. As $c_1 = d$ this happens precisely if the procedure shows that among $\{2,3,\ldots,n\}$ at least $d$ numbers were not boosted. For the sake of contradiction, assume that we have at least $n-d$ boosted numbers, and let the largest $n-d$ be $n+1-i_1 > n+1-i_2 > \ldots > n+1-i_{n-d}$. Let $T = \{i_1,\ldots,i_{n-d}\}$ then by a similar argument to before we have $(c_1,c_2,\ldots,c_n) \geq (d,\gamma_{n-1},\ldots,\gamma_1)$ where $(\gamma_{n-1},\ldots,\gamma_1) = \gamma(T)^*$.
        \item Assume that the inverse map successfully inserted all the numbers in $[j-1]$ (with $j \geq 2$) and that we now try to insert $j$ according to $c_j$. 
        
        First assume that $c_j = t$ is unboosted. Since this is unboosted, there are still at least $t$ unboosted numbers among $\{c_{j+1},\ldots,c_n\}$. As so far only indices corresponding to unboosted numbers have been inserted in empty blocks, and the number of total blocks it the number of unboosted numbers, we have at least $t+1$ empty blocks at this point. As a result, there will be some empty block labeled with $t$, so we can insert $j$ into an empty block, as desired.
        
        Hence, assume that $c_j$ was boosted. Suppose that at the time we still have $t$ nonempty blocks, then by the unboosting procedure we know that $c_j \geq t+1$, so we can insert $j$ appropriately, unless $c_j$ is too big. In other words, the only thing that can go wrong is that there were $\ell$ unboosted numbers (hence $\ell$ blocks in the ordered set partition), but that $c_j \geq \ell+1$. Let $n+1-i_1 > \ldots > n+1-i_a > j > n - i_{a+1} > \ldots > n-i_{n-\ell-1}$ be all the boosted numbers. But then, for $T = \{i_1,\ldots,i_{n-\ell-1}\}$ of size $n - (\ell+1)$, with $\gamma(T)^* = (\gamma_{n-1},\ldots,\gamma_1)$, we have $(c_1,\ldots,c_n) \geq (\gamma_{n-1},\ldots,\gamma_{n-j+1},\ell+1,\gamma_{n-j},\ldots,\gamma_1)$, a contradiction. \qedhere
        \end{itemize}
\end{itemize}
\end{proof}

\section{The algebraic quotient} \label{sec-algebra}

Recall from
the introduction that a word 
$w = w_1w_2\cdots w_n$ on the alphabet $\ZZ_{>0}$ is \emph{packed} if whenever $i+1$ appears, then so does $i$. 
It will be convenient for our inductive arguments to consider packed words in which every letter in some segment $1 \leq i \leq k$ must appear. 
To this end, we define
\begin{equation}
    \WWW_{n,k} := \{ \text{length $n$ packed words $w = w_1 w_2 \dots w_n$ \,:\, the letters $1, 2, \dots, k$ appear in $w$} \}.
\end{equation}
Words in $\WWW_{n,k}$ are in bijection with ordered set partitions of $[n]$ with at least $k$ blocks.
We have the further identifications $\WWW_{n,1} = \WWW_n$ and $\WWW_{n,n} = \symm_n.$

The symmetric group $\symm_n$ acts on $\WWW_{n,k}$ by the rule 
$\sigma \cdot (w_1 \dots w_n) := w_{\sigma(1)} \dots w_{\sigma(n)}$.
The quotient rings $S_{n,k}$ of the following definition will give a graded refinement of this action. 
Their defining ideals $J_{n,k}$ contain the ideal $J_n$ defining the ring $S_n$ appearing in the introduction.

\begin{definition}
Let $J_{n,k} \subseteq \QQ[\xx_n]$ be the ideal
\[
J_{n,k} := J_n + \langle e_n ,e_{n-1}, \ldots, e_{n-k+1} \rangle
\]
and let $S_{n,k} := \mathbb{Q}[\xx_n]/J_{n,k}$ be the corresponding quotient ring.
\end{definition}

Each of the quotients $S_{n,k}$ is a graded $\symm_n$-module. Their defining ideals are nested according to
$J_n = J_{n,n} \subseteq J_{n,n-1} \subseteq \cdots \subseteq J_{n,1}$.
We study $S_{n,k}$ by making use of a point locus $X_{n,k} \subseteq \QQ^n$ corresponding to $\WWW_{n,k}.$
Fix $n$ distinct rational numbers $\alpha_1, \dots, \alpha_n \in \QQ$. For any packed word $w_1 \dots w_n \in \WWW_{n,k}$, we have 
a corresponding point $(\alpha_{w_1}, \dots, \alpha_{w_n}) \in \QQ^n$. We let $X_{n,k} \subseteq \QQ^n$ be the family of points corresponding
to all packed words in $\WWW_{n,k}.$

The set $X_{n,k} \subseteq \QQ^n$ is closed under the coordinate-permuting action of $\symm_n$ and we have an identification
$\QQ[\WWW_{n,k}] \cong \QQ[X_{n,k}]$.
As explained in the introduction, we have isomorphisms of ungraded $\symm_n$-modules
\begin{equation*}
    \QQ[\WWW_{n,k}] \cong \QQ[X_{n,k}] \cong \QQ[\xx_n]/\mathbf{I}(X_{n,k}) \cong \QQ[\xx_n]/\mathbf{T}(X_{n,k}).
\end{equation*}
It turns out that $\mathbf{T}(X_{n,k})$ coincides with $J_{n,k}$.

\begin{theorem}
\label{thm-ungraded-structure}
    For any $1 \leq k \leq n$, we have the ideal equality $J_{n,k} = \mathbf{T}(X_{n,k})$.
    Consequently, we have an isomorphism of ungraded $\symm_n$-modules
    $\QQ[\WWW_{n,k}] \cong S_{n,k}$.
\end{theorem}

\begin{proof}
    To show that $J_{n,k} \subseteq \mathbf{T}(X_{n,k})$, it suffices to show that every generator of $J_{n,k}$ arises as the highest degree component 
    of some polynomial in $\mathbf{I}(X_{n,k}).$
    Fix $1 \leq i\leq n$ and $1 \leq r \leq d$; we begin by showing that the generator 
     $x_i^d e_{n-r}^{(i)}$ lies in $\mathbf{T}(X_{n,k}).$

Note that if $(x_1,\ldots,x_n) \in X_{n,k}$, we either have $x_i \in \{\alpha_1,\ldots,\alpha_d\}$, or for any $1 \leq j \leq d$ the number $\alpha_d$ must appear among $\{x_1,\ldots,x_{i-1},x_{i+1},\ldots,x_n\}$. We let $t$ be a new variable, and define the function
\[
f(x_1,\ldots,x_n,t) := (x_i-\alpha_1) \cdots (x_i-\alpha_d) \cdot \frac{(1-tx_1)\cdots(1-tx_{i-1})(1-tx_{i+1})\cdots(1-tx_n)}{(1-t\alpha_1)\cdots(1-t\alpha_d)}
\]
and expanding this function in terms of the parameter $t$ yields
\begin{align*}
    f(x_1,&\ldots,x_n,t) = \\
    &(x_i-\alpha_1) \cdots (x_i-\alpha_d) \cdot \sum_{r \geq 0} \left(\sum_{a+b=r} (-1)^a e_a^{(i)} \cdot h_b(\alpha_1,\ldots,\alpha_d)\right) t^r
\end{align*}
Specialization of $f(x_1,\ldots,x_n,t)$ at $(x_1,\ldots,x_n) = (\beta_1,\ldots,\beta_n)$ yields an element of $\QQ[[t]]$. 
We analyze this specialization when $(\beta_1,\ldots,\beta_n) \in X_{n,k}$. 
If $\beta_i \in \{\alpha_1,\ldots,\alpha_d\}$, then  $f(\beta_1,\ldots,\beta_n,t) = 0$. Otherwise, 
$d$ of the terms in the numerator of $f$ will cancel with the $d$ terms in the denominator, so that
hence $f(\beta_1,\ldots,\beta_n,t)$ is a polynomial of degree $(n-1)-d$ in $t$. 
Either way, the coefficient of $t^{n-r}$ in $f(x_1,\ldots,x_n,t)$ vanishes on $X_{n,k}$, so that
\[
(x_i-\alpha_1) \cdots (x_i-\alpha_d) \cdot \left(\sum_{a+b=n-r} (-1)^a e_a^{(i)} \cdot h_b(\alpha_1,\ldots,\alpha_d)\right) \in \mathbf{I}(X_{n,k})
\]
and taking the highest degree component gives
\[
x_i^d \cdot (-1)^{n-r} e_{n-r}^{(i)} \in \mathbf{T}(X_{n,k}).
\]

The remaining generators $e_d$ (for $d > n-k$) are handled by a similar argument. We consider the rational function
\begin{align*}
g(x_1, \dots, x_n, t) &:= \frac{(1-t x_1)(1 - t x_2) \cdots (1 - t x_n)}{(1 - t \alpha_1)(1 - t\alpha_2) \cdots (1 - t \alpha_k) } \\
&= \sum_{r \geq 0}  \left(
\sum_{a + b = r} (-1)^a e_a \cdot h_b(\alpha_1, \dots, \alpha_k)
\right) \cdot t^r.
\end{align*}
Evaluating $(x_1, \dots, x_n)$ at a point in $X_{n,k}$ forces the $k$ factors in the denominator to cancel with $k$ factors in the numerator,
yielding a polynomial of degree $n-k$ in $t$. For any $d > n-k$, we conclude that
\[
\sum_{a + b = d} (-1)^a e_a \cdot h_b(\alpha_1, \dots, \alpha_k) \in \mathbf{I}(X_{n,k}),
\]
which implies
\[
e_d \in \mathbf{T}(X_{n,k}).
\]
This proves the containment $J_{n,k} \subseteq \mathbf{T}(X_{n,k}),$ so that
\begin{equation}
\label{inequality-chain}
    \dim \QQ[\xx_n]/J_{n,k} \geq \dim \QQ[\xx_n]/\mathbf{T}(X_{n,k}) = |\WWW_{n,k}|
\end{equation}

In light of Equation~\eqref{inequality-chain},  to prove the desired equality $J_{n,k} = \mathbf{T}(X_{n,k})$ it is enough to show that
$\dim(\QQ[\xx_n]/J_{n,k}) \leq |\WWW_{n,k}|$. This is a Gr\"obner theory argument. 

Since the elementary symmetric polynomials
$e_n, e_{n-1}, \dots, e_{n-k+1}$ in the full variable set $\{x_1, \dots, x_n\}$ lie in $J_{n,k}$,
\cite[Lem. 3.4]{HRS} implies that for any subset $S \subseteq [n]$ with $|S| = n-k+1$, the {\em Demazure character} 
$\kappa_{\gamma(S)}$ corresponding to the length $n$ sequence $\gamma(S)$ also lies in $J_{n,k}$. The lexicographical leading 
monomial of $\kappa_{\gamma(S)}$ has exponent sequence $\gamma(S)^*$. Similarly, for $1 \leq i, d \leq n$, since 
\begin{equation*}
    x_i^d \cdot e_{n-d}^{(i)}, \dots, x_i^d \cdot e_{n-1}^{(i)} \in J_{n,k}, 
\end{equation*}
for any $T \subseteq [n-1]$ of size $|T| = n-d$, \cite[Lem. 3.4]{HRS} again implies that 
\begin{equation*}
    x_i^d \cdot \kappa_{\gamma(T)}(x_1, \dots, x_{i-1}, x_{i+1}, \dots, x_n) \in J_{n,k}.
\end{equation*}
Writing $\gamma(T)^* = (\gamma_{n-1}, \dots, \gamma_1)$, the lexicographical leading term of 
$x_i^d \cdot \kappa_{\gamma(T)}(x_1, \dots, x_{i-1}, x_{i+1}, \dots, x_n)$ is 
$(\gamma_{n-1}, \dots, \gamma_i, d, \gamma_{i-1}, \dots, \gamma_1)$. It follows that
\begin{quote}
    the exponent sequence $(c_1, \dots, c_n)$  of any member of the standard monomial basis of $\QQ[\xx_n]/J_{n,k}$
    satisfies the conditions in the statement of Theorem~\ref{thm-combinatorial-bijection}.
\end{quote}
Theorem~\ref{thm-combinatorial-bijection} implies the desired dimension bound $\dim \QQ[\xx_n]/J_{n,k} \leq |\WWW_{n,k}|$, completing the proof.
\end{proof}

The standard monomial basis of $S_{n,k}$ is governed by coinversion codes.

\begin{corollary}
    The standard monomial basis of $S_{n,k}$ with respect to the lexicographical term ordering are the monomials
    $x_1^{c_1} \cdots x_n^{c_n}$ where $(c_1, \dots, c_n) = \overline{\mathbf{code}}(\sigma)$ is the boosted coinversion code
    of some ordered set partition $\sigma$ of $[n]$ with at least $k$ blocks.
\end{corollary}

\begin{proof}
    This follows from Theorem~\ref{thm-combinatorial-bijection} and the last paragraph of the above proof.
\end{proof}







Our next goal is to derive the graded $\symm_n$-module structure of the quotients $S_{n,k}$.
This result is stated most cleanly in terms of the following rings defined by 
Haglund, Rhoades, and Shimozono \cite{HRS}.

\begin{definition} \label{rnk-quotient-definition}
Let $1 \leq k \leq n$ be integers. Define the ideal $I_{n,k} \subseteq \QQ[\xx_n]$ by
\[
I_{n,k} := \langle x_1^k, x_2^k, \ldots, x_n^k, e_n, e_{n-1}, \ldots, e_{n-k+1} \rangle
\]
and let $R_{n,k} := \QQ[\xx_n]/I_{n,k}$ be the corresponding quotient ring.
\end{definition}

We can now state the graded $\symm_n$-module structure of $S_n$ in terms of the graded $\symm_n$-module structure of these $R_{n,k}$, which have been extensively studied in \cite{HRS}.

\begin{theorem} \label{thm-graded-module-refined}
As graded $\symm_n$-module we have
\[
S_{n,k} \cong R_{n,n}\langle 0 \rangle \oplus R_{n,n-1}\langle -1 \rangle \oplus \cdots \oplus R_{n,k}\langle -n+k \rangle.
\]
\end{theorem}

We are now ready to prove Theorem \ref{thm-graded-module-refined}.

\begin{proof}
We proceed by descending induction on $k$.
In the case $n = k$, we claim that $J_{n,n} = I_{n,n} = \langle e_1, \dots, e_n \rangle$ is the classical invariant ideal
so that $S_{n,n} = R_{n,n}$. Indeed, each elementary symmetric polynomial $e_d$ appears as a generator of $J_{n,n}$.
On the other hand, Theorem~\ref{thm-ungraded-structure} implies that $\dim S_{n,n} = n! = \dim R_{n,n}$. This finishes the proof in the case $k = n$.

Now suppose $1 \leq k \leq n-1$.
We exhibit a short exact sequence of $\symm_n$-modules
\begin{equation}
    \label{exact-sequence}
    0 \rightarrow R_{n,k} \overset{\varphi}{\rightarrow} S_{n,k} \overset{\pi}{\rightarrow} S_{n,k+1} \rightarrow 0,
\end{equation}
where $\varphi$ is homogeneous of degree $n-k$ and $\pi$ is homogeneous of degree $0$. The exactness of this sequence implies
\[
S_{n,k} \cong S_{n,k+1} \oplus R_{n,k}\langle -n+k \rangle,
\]
proving the theorem by induction.

Since every generator of $J_{n,k+1}$ is also a generator of $J_{n,k}$, we may take $\pi: S_{n,k} \twoheadrightarrow S_{n,k+1}$ to be the canonical projection.
We have a map
\begin{equation}
    \widetilde{\varphi}: \QQ[\xx_n] \rightarrow S_{n,k}
\end{equation}
given by multiplication by $e_{n-k}$ followed by projection onto $S_{n,k}$. 
We verify that $\widetilde{\varphi}$ descends to a map $\varphi: R_{n,k} \rightarrow S_{n,k}$ by showing that $\widetilde{\varphi}$ sends every generator of $I_{n,k}$ to zero.
Indeed, we have
$\widetilde{\varphi}(e_j(x_1, \dots, x_n)) = 0$ for any $j > n-k$ since $e_j(x_1, \dots, x_n)$ is a generator of $J_{n,k}$. 
Furthermore, for $1 \leq i \leq n$ we have
\[
\widetilde{\varphi}(x_i^k) = 
x_i^k e_{n-k} = x_i^k e_{n-k}^{(i)} + 
x_i^{k+1} e_{n-k-1}^{(i)} = 0,
\]
where the final equality follows because both 
$x_i^{k} e_{n-k}^{(i)}$ and 
$x_i^{k+1} e_{n-k-1}^{(i)}$ are generators of $J_{n,k}$.
We conclude that $\widetilde{\varphi}$ descends to a map $\varphi: R_{n,k} \rightarrow S_{n,k}$ of $\symm_n$-modules which is homogeneous of degree $n-k$.
It is clear that $\varphi$ surjects onto the kernel of $\pi$.
The exactness of the sequence \eqref{exact-sequence} follows from the dimensional equality
\[
\dim(S_{n,k}) = |\WWW_{n,k}| = |\WWW_{n,k+1}| + |\OP_{n,k}| = \dim(S_{n,k+1}) + \dim(R_{n,k}). \qedhere
\]
\end{proof}

The graded Frobenius image of $S_{n,k}$ is most naturally stated in terms of the $C$-functions
defined in Equation~\eqref{c-function-definition}.

\begin{corollary}
For any $1 \leq k \leq n$, the graded Frobenius image of $S_{n,k}$ is given by 
\begin{equation}
    \grFrob(S_{n,k}; q) = \sum_{j = k}^n q^{n-j} \cdot (\omega \circ \mathrm{rev}_q) C_{n,j}(\xx;q).
\end{equation}
\end{corollary}

\begin{proof}
    Apply \cite[Thm. 6.11]{HRS} and Theorem~\ref{thm-graded-module-refined}.
\end{proof}

\section{Conclusion} \label{sec-conclusion}

In this paper we have described an quotient $S_n$ of $\QQ[\xx_n]$ whose algebraic properties are governed by the combinatorics of packed words in $\WWW_n$.
The ring $S_n$ has provable algebraic properties which are similar to conjectural properties of the superspace coinvariant ring  
$\Omega_n / \langle (\Omega_n)^{\symm_n}_+ \rangle$. 
With an eye towards proving these conjectures, it would be desirable to have a more direct connection between the packed word quotient $S_n$ and the 
superspace coinvariant ring.

Generalized coinvariant rings related to delta operators have seen ties to cohomology theory.
In the context of the rings $R_{n,k}$ of Definition~\ref{rnk-quotient-definition}, Pawlowski and Rhoades showed that 
$H^{\bullet}(X_{n,k};\QQ) = R_{n,k}$, where $X_{n,k}$ is the variety of $n$-tuples $(\ell_1, \dots, \ell_n)$ of 1-dimensional subspaces of $\CC^k$ which satisfy
$\ell_1 + \cdots + \ell_n = \CC^k$.
Rhoades and Wilson \cite{RW} refined this result by considering the open subvariety $X^{(r)}_{n,k}$ obtained by requiring the the 
sum $\ell_1 + \cdots + \ell_r$ of the first $r$ lines is 
direct.
In light of \cite{PR, RW}, it is natural to ask for a geometric perspective on the ring $S_n$ appearing in this paper.

\begin{problem}
\label{geometry-problem}
Find a variety $Y_n$ whose rational cohomology ring $H^{\bullet}(Y_n; \QQ)$ is isomorphic to $S_n$.
\end{problem}

The results in \cite{PR} suggest that $Y_n$ could be taken to be an open subvariety of the $n$-fold Cartesian product $(\mathbb{P}^{k-1})^n$ of
$(k-1)$-dimensional projective space with itself with the property that the cohomology map
$i^*: H^{\bullet}((\mathbb{P}^{k-1})^n;\QQ) \rightarrow H^{\bullet}(Y_n;\QQ)$ induced by the inclusion
$i: Y_n \hookrightarrow (\mathbb{P}^{k-1})^n$ is surjective.

\section{Acknowledgements}

B. Rhoades was partially supported by NSF Grant DMS-1953781.
The authors are grateful for Christopher O'Neill for helpful discussions about this project.

\end{document}